\documentclass[11pt]{article}

\usepackage[a4paper,margin=1in]{geometry}
\usepackage{amsmath,amssymb,amsthm,mathtools}
\usepackage{mathrsfs}
\usepackage{aliascnt}
\usepackage{microtype}
\usepackage{hyperref}
\usepackage[nameinlink,capitalize,noabbrev]{cleveref}

\hypersetup{
  colorlinks=true,
  linkcolor=blue,
  citecolor=red,
  urlcolor=blue,
  pdftitle={A Coulomb-Corrected Labeled Energy and Growth Estimates for the Vlasov–Poisson System},
  pdfauthor={Zhaopeng Wang}
}

\numberwithin{equation}{section}
\allowdisplaybreaks[2]

\theoremstyle{plain}
\newtheorem{theorem}{Theorem}[section]

\newaliascnt{proposition}{theorem}
\newtheorem{proposition}[proposition]{Proposition}
\aliascntresetthe{proposition}

\newaliascnt{lemma}{theorem}
\newtheorem{lemma}[lemma]{Lemma}
\aliascntresetthe{lemma}

\newaliascnt{corollary}{theorem}

\aliascntresetthe{corollary}

\theoremstyle{definition}
\newaliascnt{definition}{theorem}

\aliascntresetthe{definition}

\theoremstyle{remark}
\newaliascnt{remark}{theorem}
\newtheorem{remark}[remark]{Remark}
\aliascntresetthe{remark}

\crefname{theorem}{Theorem}{Theorems}
\crefname{proposition}{Proposition}{Propositions}
\crefname{lemma}{Lemma}{Lemmas}
\crefname{corollary}{Corollary}{Corollaries}
\crefname{definition}{Definition}{Definitions}
\crefname{remark}{Remark}{Remarks}
\Crefname{theorem}{Theorem}{Theorems}
\Crefname{proposition}{Proposition}{Propositions}
\Crefname{lemma}{Lemma}{Lemmas}
\Crefname{corollary}{Corollary}{Corollaries}
\Crefname{definition}{Definition}{Definitions}
\Crefname{remark}{Remark}{Remarks}

\newcommand{\R}{\mathbb{R}}
\newcommand{\dd}{\,\mathrm d}
\newcommand{\supp}{\operatorname{supp}}
\newcommand{\one}{\mathbf 1}
\newcommand{\abs}[1]{\left|#1\right|}
\newcommand{\norm}[1]{\left\lVert#1\right\rVert}
\newcommand{\cH}{\mathcal H}
\newcommand{\cJ}{\mathcal J}
\newcommand{\cD}{\Delta}
\newcommand{\LE}{\mathscr E}
\newcommand{\cG}{\mathcal G}
\newcommand{\cB}{\mathcal B}
\newcommand{\cU}{\mathcal U}
\newcommand{\tw}{\widetilde w}
\newcommand{\ta}{\widetilde a}
\newcommand{\eps}{\varepsilon}
\newcommand{\tscale}[1]{(1+#1)}
\newcommand{\tlog}[1]{\log(2+#1)}
\newcommand{\Tlog}{\log(2+T)}

\title{A Coulomb-Corrected Labeled Energy and Growth Estimates for the Vlasov–Poisson System}

\author{
Zhaopeng Wang\\[0.4em]
\small School of Mathematics and Statistics\\
\small Beijing Institute of Technology\\
\small Beijing 100081, China\\[0.3em]
\small\texttt{zpwang@bit.edu.cn}
}
\date{}

\begin{document}
\maketitle

\begin{abstract}
This paper introduces a Coulomb-corrected labeled energy for smooth compactly
supported solutions of the three-dimensional repulsive Vlasov--Poisson
system. The resulting identity replaces the direct electric-field term along
a characteristic by a weighted singular integral, which can be estimated
more effectively. Thus this structure yields improved bounds on the velocity support and the decay
of the electric field.
\end{abstract}



\section{Introduction}\label{sec:introduction}

\subsection{Background and previous results}\label{subsec:background}

The Vlasov--Poisson system is a basic kinetic model for a collisionless ensemble
of particles interacting through a self-consistent field. We study the
three-dimensional plasma, or repulsive, case in the whole space: particles of
the same charge repel one another, and the resulting cloud tends to disperse.
At the mathematical level, the phase-space density is transported along
characteristics, while its spatial density generates a nonlocal Coulomb field
with an inverse-square singularity. More precisely, \(f=f(t,x,v)\ge0\)  denotes the phase-space density of particles at time
\(t\in\R^+\), position \(x\in\R^3\), velocity \(v\in\R^3\) and solves
\begin{equation}\label{eq:VP}
\begin{aligned}
 &\partial_t f+v\cdot\nabla_x f+E\cdot\nabla_v f=0,
 \qquad f|_{t=0}=f_0,\\
 &\rho(t,x)=\int_{\R^3}f(t,x,v)\dd v,\\
 &E(t,x)=\frac1{4\pi}\int_{\R^3}
 \frac{x-y}{\abs{x-y}^3}\rho(t,y)\dd y.
\end{aligned}
\end{equation}
We denote the positive Coulomb potential by
\[
 U(t,x):=\frac1{4\pi}\int_{\R^3}
 \frac{\rho(t,y)}{\abs{x-y}}\dd y,
 \qquad -\Delta_xU=\rho,
 \qquad E=-\nabla_xU.
\]
Thus \(U\geq0\), while the relation \(E=-\nabla_xU\) gives the repulsive Coulomb force. The repulsive structure provides a positive interaction energy and strong dispersive estimates, but these estimates are essentially averaged and do not directly control the fastest characteristics in the support. Since the Coulomb force is both long ranged and singular, obtaining a uniform bound on the velocity support from such dispersive information remains a delicate problem.

Throughout the paper, the initial datum satisfies
\begin{equation}
 0\le f_0\in C_c^1(\R_x^3\times\R_v^3). \label{eq:data}
\end{equation}
Let \(f\) denote the corresponding global classical solution. Global
classical solvability in three dimensions for general smooth, compactly
supported data follows from the foundational works of Lions and Perthame,
Pfaffelmoser, and Schaeffer; see
\cite{LionsPerthame,Pfaffelmoser,SchaefferGlobal} and the survey
\cite{ReinSurvey}.

For \(t\ge0\), define the velocity-support radius
\[
 R_f(t):=\sup\bigl\{\abs{v}:(x,v)\in\supp f(t)\bigr\}
\]
and the self-similar velocity
\[
 w(t,x,v):=v-\frac{x}{1+t},
 \qquad
 \theta(t):=\sup_{(x,v)\in\supp f(t)}\abs{w(t,x,v)}.
\]
For the trivial solution, we set both suprema equal to zero; the theorem is
then immediate. Henceforth, assume \(f_0\not\equiv0\). The large-time growth
of \(R_f\) has been studied through a sequence of refinements of the
characteristic method; see, among others,
\cite{ChenLi,ChenZhang,Horst,PallardNote,PallardKRM,PallardJDE,ReinGrowth}.
A central difficulty is that a pointwise field estimate does not capture the
geometry of source characteristics that remain close to a fast target
characteristic. By combining dispersive information with an adaptive
phase-space crossing argument, Chen and Li \cite{ChenLi} obtained
\begin{equation}
 R_f(t)+\theta(t)
 \le C(1+t)^{2/15}\log^{8/15}(2+t). \label{eq:chen-li-support}
\end{equation}
More precisely, their theorem gives the stated bound for \(R_f\); the same
bound for \(\theta\) follows from the characteristic equation and the compactness of the initial support.

Self-similar dispersive identities and global Lyapunov functionals for the
repulsive system have a substantial history; see, for example,
\cite{ChaeHa,IllnerRein,Perthame}. These global structures provide important information on the averaged dispersive behavior of the solution. However, the velocity-support problem is essentially pointwise: it requires control of the fastest characteristics in the support, and such control does not follow directly from global energy or moment estimates.

The purpose of the present paper is to develop a characteristic-level refinement of these dispersive structures and use it to improve the estimate \eqref{eq:chen-li-support}. More precisely, we introduce a Coulomb-corrected labeled energy associated with an individual characteristic. This quantity is designed to organize the repulsive interaction in a form better suited to the control of extreme velocities. It allows the dispersive information of the solution to be combined more effectively with the geometry of nearby characteristics, leading to a sharper growth estimate for the velocity support and, consequently, to an improved decay estimate for the electric field.

\subsection{Main result}\label{subsec:main-result}

The main theorem gives simultaneous control of the velocity-support radius
and the electric field.

\begin{theorem}[Velocity support and field decay]\label{thm:main}
Let \(f\) be the classical solution of \eqref{eq:VP} corresponding to an
initial datum satisfying \eqref{eq:data}. Then there exists a constant
\(C>0\), depending only on \(f_0\), such that for every \(t\ge0\),
\begin{equation}
 R_f(t)\le C(1+t)^{2/21}\log^{4/21}(2+t). \label{eq:main-bound}
\end{equation}
Moreover,
\begin{equation}
 \norm{E(t)}_{L^\infty(\R^3)}
 \le C(1+t)^{-13/63}\log^{16/63}(2+t). \label{eq:field-decay}
\end{equation}
\end{theorem}
\begin{remark}
The exponent \(2/21\) improves the exponent \(2/15\) in
\eqref{eq:chen-li-support}. Once the support bound is known, the field decay
follows from the standard Coulomb interpolation estimate. Thus the main
analysis is concentrated in the first conclusion of \cref{thm:main}. To the best of our knowledge, \eqref{eq:main-bound} gives the best currently known growth estimate for the velocity support for the three-dimensional repulsive
Vlasov--Poisson system with general smooth, compactly supported initial data.
\end{remark}

\subsection{The labeled-energy mechanism and proof strategy}
\label{subsec:strategy}
For a characteristic \((X_*(t),V_*(t))\), write
\(w_*(t):=V_*(t)-X_*(t)/(1+t)\), the restriction of \(w\) to that
characteristic. Define the labeled energy at a phase point by
\begin{equation}
 \LE(t,x,v):=\frac12\abs{w(t,x,v)}^2+U(t,x).
 \label{eq:labeled-energy-def}
\end{equation}
Its value along the selected characteristic will be denoted by
\[
 \LE_*(t):=\LE(t,X_*(t),V_*(t))
 =\frac12\abs{w_*(t)}^2+U(t,X_*(t)).
\]
The point of this definition is the exact identity
\begin{equation}
 \frac{\dd}{\dd t}\LE_*(t)
 +\frac{\abs{w_*(t)}^2+U(t,X_*(t))}{1+t}
 =\cJ_*(t), \label{eq:intro-identity}
\end{equation}
where
\begin{equation}
 \cJ_*(t):=\frac1{4\pi}\iint_{\R^3\times\R^3}
 w(t,x,v)\cdot
 \frac{X_*(t)-x}{\abs{X_*(t)-x}^3}
 f(t,x,v)\dd x\dd v. \label{eq:J-def}
\end{equation}
Indeed, the derivative of the self-similar kinetic energy produces
\(w_*(t)\cdot E(t,X_*(t))\), whereas the derivative of the Coulomb correction
produces the same term with the opposite sign. The bare force is therefore
replaced by the weighted current \(\cJ_*\). This cancellation is the
structural starting point of the proof.

The analytic part of the argument keeps the factor \(w\) inside the singular
integral until the final radial integration. In the quadratic \(w\)-branch
this removes an inverse-cutoff loss; in the linear \(w\)-branch it removes an
unnecessary logarithmic loss. To select the target characteristic
consistently, only one additional global quantity is needed: the running
maximum of the full labeled energy,
\begin{equation}
 \cH(t):=
 \max_{\substack{0\le s\le t\\(x,v)\in\supp f(s)}}
 \left\{\frac12\abs{w(s,x,v)}^2+U(s,x)\right\}. \label{eq:envelope-def}
\end{equation}
At a record time \(T\), a target characteristic realizes the maximum in
\eqref{eq:envelope-def}; set \(Q=\cH(T)^{1/2}\). On an adaptive record interval
\([T-\delta,T]\), the weighted
\(\cG\)--\(\cB\)--\(\cU\) decomposition developed in
\cref{sec:current}, followed by the one-parameter optimization in
\cref{sec:conclusion}, gives
\[
 \int_{T-\delta}^{T}\abs{\cJ_*(s)}\dd s
 \le C\delta\cH(T)^{1/4}
 (1+T)^{-6/7}\log^{2/7}(2+T).
\]
Here the construction of \(\delta\) uses the adaptive-time lower bound
recorded by Chen and Li \cite{ChenLi}. Combining this estimate with
\eqref{eq:intro-identity} and iterating over
those record intervals gives
\[
 \cH(t)\le C(1+t)^{4/21}\log^{8/21}(2+t),
\]
which implies \eqref{eq:main-bound}.

\paragraph{Notation and organization.}
The letter \(C\) denotes a positive constant that may change from line to
line and depends only on \(f_0\) and fixed numerical choices. Constants with
subscripts remain fixed within the corresponding arguments. For two positive
quantities, \(A\asymp B\) means that \(A/B\) is bounded above
and below by positive constants of this kind. The symbol
\(\one_{\mathcal A}\) denotes the indicator of a set \(\mathcal A\).

\Cref{sec:preliminaries} develops the full framework needed later: the exact
labeled-energy identity, the standard dispersive and potential estimates, the
adaptive impulse scale, the crossing tools, and the running record envelope.
Its role is therefore preparatory rather than the proof of a single
proposition. By contrast, \cref{sec:current,sec:conclusion} state their main
outputs at the beginning. \Cref{sec:current} proves the weighted-current
estimate \cref{prop:local-current} by the
\(\cG\)--\(\cB\)--\(\cU\) decomposition, while
\cref{sec:conclusion} closes the labeled-energy envelope by optimization and
record-time iteration and then deduces \cref{thm:main}.

\section{The labeled-energy framework}\label{sec:preliminaries}

This section develops all ingredients needed to formulate and iterate the
labeled energy on a record interval. We first prove the exact identity that
replaces the electric field along a target characteristic by the weighted
singular current \(\cJ_*\). We then collect the dispersive and potential
bounds, recall the adaptive impulse estimate, establish the uniform reduction
of the source-dependent time scale, record the geometric crossing tools, and
construct the running labeled-energy envelope. Later sections use these
ingredients together.

Let
\[
 (X(s;t,x,v),V(s;t,x,v))
\]
be the characteristic ending at \((x,v)\) at time \(t\):
\begin{equation*}
 \dot X(s)=V(s),
 \qquad
 \dot V(s)=E(s,X(s)),
 \qquad
 (X(t),V(t))=(x,v).
\end{equation*}

When a characteristic is selected as the target, we write it as
\((X_*,V_*)\) and use the notation introduced in
\eqref{eq:labeled-energy-def}--\eqref{eq:J-def}.

\subsection{The exact labeled-energy identity}

\begin{proposition}[Exact labeled-energy identity]\label{prop:labeled-identity}
Along every characteristic \((X_*,V_*)\),
\begin{equation}
 \frac{\dd}{\dd t}\LE_*(t)
 +\frac{\abs{w_*(t)}^2+U(t,X_*(t))}{\tscale{t}}
 =\cJ_*(t). \label{eq:labeled-identity}
\end{equation}
In particular,
\begin{equation}
 \frac{\dd}{\dd t}\LE_*(t)
 +\frac{\LE_*(t)}{\tscale{t}}
 \le\abs{\cJ_*(t)}. \label{eq:labeled-inequality}
\end{equation}
\end{proposition}

\begin{proof}
All differentiations below are justified by the classical regularity and
compact support of \(f\).
Since
\[
 w_*'=E(t,X_*)-\frac{w_*}{\tscale{t}},
\]
we have
\begin{equation}
 \frac12\frac{\dd}{\dd t}\abs{w_*}^2
 =w_*\cdot E(t,X_*)-\frac{\abs{w_*}^2}{\tscale{t}}. \label{eq:q-energy}
\end{equation}
Let \(K(z)=(4\pi)^{-1}z\abs{z}^{-3}\). The continuity equation
\(\partial_t\rho+\nabla_x\cdot j=0\), with \(j=\int vf\dd v\), gives
\begin{align*}
 \frac{\dd}{\dd t}U(t,X_*(t))
 &=\partial_tU(t,X_*)+V_*\cdot\nabla_xU(t,X_*)\\
 &=\iint_{\R^3\times\R^3}
 (v-V_*)\cdot K(X_*-x)f(t,x,v)\dd x\dd v.
\end{align*}
Using
\[
 V_*-v=w_*-w+\frac{X_*-x}{\tscale{t}},
\]
we obtain
\begin{equation}
 \frac{\dd}{\dd t}U(t,X_*)
 =-w_*\cdot E(t,X_*)+\cJ_*(t)
 -\frac{U(t,X_*)}{\tscale{t}}. \label{eq:U-energy}
\end{equation}
Adding \eqref{eq:q-energy} and \eqref{eq:U-energy} proves
\eqref{eq:labeled-identity}. Since
\[
 \abs{w_*(t)}^2+U(t,X_*(t))
 \ge\frac12\abs{w_*(t)}^2+U(t,X_*(t))=\LE_*(t),
\]
\eqref{eq:labeled-inequality} follows.
\end{proof}

\subsection{Known dispersive estimates and basic potential bounds}

The characteristic flow preserves the Lebesgue measure and
\begin{equation}
 f(s,X(s;t,x,v),V(s;t,x,v))=f(t,x,v). \label{eq:f-characteristics}
\end{equation}
We use the following estimates from the repulsive theory.

\begin{proposition}[Known dispersive and velocity-support estimates]\label{prop:known-estimates}
There exists \(C>0\) such that for every \(t\ge0\),
\begin{align}
 \norm{f(t)}_{L^1_{x,v}}+\norm{f(t)}_{L^\infty_{x,v}}&\le C, \notag\\
 \iint_{\R^3\times\R^3}\abs{w(t,x,v)}^2f(t,x,v)\dd x\dd v
 &\le C\tscale{t}^{-1}, \label{eq:w2}\\
 \norm{\rho(t)}_{L^{5/3}_x}&\le C\tscale{t}^{-3/5}, \label{eq:rho53}
\end{align}
and the velocity-support estimate \eqref{eq:chen-li-support} holds.
\end{proposition}

\begin{proof}
The shifted second-moment estimate \eqref{eq:w2} is the \(\alpha=1\) case of
the \(\alpha>0\) estimate proved in \cite{ChenZhangMoments}; see also
\cite[Lemma~2.1]{ChenLi}. The corresponding \(\alpha=0\) dispersive
identities go back to Illner--Rein and Perthame
\cite{IllnerRein,Perthame}, and \eqref{eq:rho53} follows from
\eqref{eq:w2} by the standard kinetic interpolation. The bound for \(R_f\)
in \eqref{eq:chen-li-support} is \cite[Theorem~1.1]{ChenLi}; as noted in the
introduction, its companion bound for \(\theta\) follows directly from the
characteristic equation.
\end{proof}

We shall repeatedly use a Coulomb interpolation estimate.

\begin{lemma}[Coulomb interpolation]\label{lem:coulomb-interpolation}
Let \(1\le p<3\) and let \(g\ge0\) belong to
\(L^p(\R^3)\cap L^\infty(\R^3)\). Then
\begin{equation}
 \sup_{y\in\R^3}\int_{\R^3}
 \frac{g(x)}{\abs{x-y}^2}\dd x
 \le C_p\norm{g}_{L^p}^{p/3}\norm{g}_{L^\infty}^{1-p/3}. \label{eq:coulomb-interpolation}
\end{equation}
\end{lemma}

\begin{proof}
Split the integral at a radius \(R>0\). The near field is controlled by the \(L^\infty\) norm:
\[
 \int_{\abs{x-y}\le R}\frac{g(x)}{\abs{x-y}^2}\dd x
 \le \norm{g}_{L^\infty}
 \int_{\abs{z}\le R}\abs{z}^{-2}\dd z
 \le C\norm{g}_{L^\infty}R.
\]
For \(1<p<3\), H\"older's inequality in the far field gives
\[
 \int_{\abs{x-y}>R}\frac{g(x)}{\abs{x-y}^2}\dd x
 \le \norm{g}_{L^p}
 \left(\int_{\abs{z}>R}\abs{z}^{-2p'}\dd z\right)^{1/p'}
 \le C_p\norm{g}_{L^p}R^{1-3/p}.
\]
For \(p=1\), the same estimate reads
\(R^{-2}\norm{g}_{L^1}\). Optimizing
\[
 C\norm{g}_{L^\infty}R
 +C_p\norm{g}_{L^p}R^{1-3/p}
\]
in \(R\) yields \eqref{eq:coulomb-interpolation}. This is the standard estimate used in \cite[Lemma~3.2]{ChenLi}.
\end{proof}

\begin{remark}
In every application below we use \eqref{eq:coulomb-interpolation} with \(p=5/3\), in which case the split estimate is
\[
 \int\frac{g(x)}{\abs{x-y}^2}\dd x
 \le C\norm{g}_{L^\infty}R
 +C\norm{g}_{L^{5/3}}R^{-4/5}.
\]
\end{remark}

The potential correction is lower order on large record scales.

\begin{lemma}[Decay of the Coulomb potential]\label{lem:U-decay}
For every \(t\ge0\),
\begin{equation}
 \norm{U(t)}_{L^\infty(\R^3)}
 \le C\tscale{t}^{-1/2}. \label{eq:U-decay}
\end{equation}
\end{lemma}

\begin{proof}
For \(r>0\), H\"older's inequality in the near field and mass conservation in the far field give
\[
 U(t,x)
 \le C\norm{\rho(t)}_{L^{5/3}}r^{1/5}+Cr^{-1}.
\]
Optimizing in \(r\) yields
\[
 \norm{U(t)}_{L^\infty}
 \le C\norm{\rho(t)}_{L^{5/3}}^{5/6}.
\]
The conclusion follows from \eqref{eq:rho53}.
\end{proof}

\subsection{The adaptive impulse scale}

For \(t>0\) and \(0\le\delta\le t\), set
\[
 F_t(\delta):=
 \max_{(x,v)\in\supp f(t)}
 \int_{t-\delta}^{t}
 \abs{E(s,X(s;t,x,v))}\dd s.
\]
For \(\lambda\ge0\), define
\begin{equation}
 \cD(t,\lambda):=
 \sup\left\{\delta\in[0,t]:F_t(\delta)\le\lambda\right\}.
 \label{eq:Delta-def}
\end{equation}

The following elementary observation makes the use of \(\cD\) at its endpoint rigorous.

\begin{lemma}[Closedness of the impulse scale]\label{lem:Delta-closed}
For each fixed \(t>0\), the function \(F_t\) is continuous and
nondecreasing. Consequently, for every \(\lambda\ge0\),
\begin{equation*}
 0\le\delta\le\cD(t,\lambda)
 \quad\Longrightarrow\quad
 F_t(\delta)\le\lambda.
\end{equation*}
In particular, the supremum in \eqref{eq:Delta-def} is a maximum.
\end{lemma}

\begin{proof}
The terminal support \(\supp f(t)\) is compact. The map
\[
 (\delta,x,v)\longmapsto
 \int_{t-\delta}^{t}
 \abs{E(s,X(s;t,x,v))}\dd s
\]
is continuous on the compact set \([0,t]\times\supp f(t)\), because the
classical field and the characteristic flow are continuous. It is therefore
uniformly continuous, and its maximum over the terminal support is continuous
in \(\delta\). Monotonicity follows from the nonnegativity of the integrand.
Hence the set
\(\{\delta:F_t(\delta)\le\lambda\}\) is a closed interval containing zero, which proves the claim.
\end{proof}

We use the following lower bound for the adaptive impulse scale.

\begin{proposition}[Adaptive-time lower bound]\label{prop:Delta}
There exists \(c_0\in(0,1)\), depending only on the initial datum, such that for every \(t>0\) and \(\lambda>0\),
\begin{equation}
 \cD(t,\lambda)\ge c_0\min\left\{
 t,\frac{\lambda^2\tscale{t}}{\tlog{t}},
 \frac{\lambda\tscale{t}^{5/7}}{\tlog{t}^{4/7}}
 \right\}. \label{eq:Delta-lower}
\end{equation}
\end{proposition}

This is \cite[Lemma~2.3]{ChenLi}, where the estimate is traced back to Pallard
\cite{PallardKRM}. The case \(\lambda=0\), when
needed below, is immediate because the right-hand side then vanishes.

The source-dependent crossing time below contains both the relative velocity
and the self-similar velocity. Its numerical prefactor must be chosen
uniformly small.

\begin{lemma}[Uniform rescaling of the local time]\label{lem:rescaling}
Fix \(0<\eps<1/20\). There exists \(c_d>0\) with the following property.
Given a target characteristic \((X_*,V_*)\), a terminal time \(t>0\), and a
terminal point \((x,v)\), put
\[
 a=v-V_*(t),
 \qquad
 w=v-\frac{x}{\tscale{t}},
\]
and define
\begin{equation*}
 \begin{aligned}
 d(t,x,v):=c_d\min\Bigg\{&t,
 \frac{\tscale{t}\abs{a}^2}{\tlog{t}},
 \frac{\tscale{t}^{5/7}\abs{a}}{\tlog{t}^{4/7}},\frac{\tscale{t}\abs{w}^2}{\tlog{t}},
 \frac{\tscale{t}^{5/7}\abs{w}}{\tlog{t}^{4/7}}
 \Bigg\}.
 \end{aligned}
\end{equation*}
Then
\begin{equation}
 d(t,x,v)\le
 \min\left\{\frac t5,
 \cD(t,\eps\abs{a}),
 \cD(t,\eps\abs{w})\right\}. \label{eq:d-Delta}
\end{equation}
Moreover, given fixed constants \(\kappa>0\) and \(C_w\ge1\), the constant
\(c_d\) can be decreased, depending only on \(\kappa,C_w,\eps\), so that,
for every \(Q>0\),
\begin{equation}
 \abs{w}\le C_wQ
 \quad\Longrightarrow\quad
 d(t,x,v)\le
 \min\left\{\frac t5,\cD(t,\kappa Q)\right\}. \label{eq:d-global}
\end{equation}
\end{lemma}

\begin{proof}
If \(a=0\), then \(d(t,x,v)=0\), and the assertion involving
\(\cD(t,\eps\abs{a})\) is immediate. Otherwise, by \cref{prop:Delta},
\[
 \cD(t,\eps\abs{a})
 \ge c_0\min\left\{t,
 \eps^2\frac{\tscale{t}\abs{a}^2}{\tlog{t}},
 \eps\frac{\tscale{t}^{5/7}\abs{a}}{\tlog{t}^{4/7}}
 \right\}.
\]
The two branches involving \(a\) in the definition of \(d\) therefore give
the corresponding part of \eqref{eq:d-Delta} if
\[
 c_d\le\min\left\{\frac15,c_0,c_0\eps^2,c_0\eps\right\}.
\]
The argument for the two branches involving \(w\) is identical, with the
case \(w=0\) again immediate. If \(\abs{w}\le C_wQ\), then
\[
 d(t,x,v)\le c_d\min\left\{t,
 C_w^2\frac{Q^2\tscale{t}}{\tlog{t}},
 C_w\frac{Q\tscale{t}^{5/7}}{\tlog{t}^{4/7}}
 \right\}.
\]
Comparison with \eqref{eq:Delta-lower} at the threshold \(\kappa Q\) proves
\eqref{eq:d-global} after additionally imposing
\[
 c_dC_w^2\le c_0\kappa^2,
 \qquad
 c_dC_w\le c_0\kappa.
\]
\end{proof}

\subsection{Geometric crossing and time extension}

We first record the elementary geometric estimate behind the
source-characteristic argument.

\begin{lemma}[Geometric crossing]\label{lem:geometric-crossing}
Let \(h>0\) and \(Z\in C^1([t-h,t];\R^3)\). Assume \(Z'(t)\ne0\) and
\[
 \abs{Z'(s)-Z'(t)}\le\eta\abs{Z'(t)}
 \quad\text{for every }s\in[t-h,t],
 \qquad 0<\eta<\frac12.
\]
If \(r_0>0\) and \(\Omega\subset[t-h,t]\) is measurable with
\(\abs{Z(s)}\ge r_0\) on \(\Omega\), then
\begin{equation*}
 \int_\Omega\frac{\dd s}{\abs{Z(s)}^2}
 \le\frac{C_\eta}{\abs{Z'(t)}r_0}.
\end{equation*}
\end{lemma}

\begin{proof}
Let \(e=Z'(t)/\abs{Z'(t)}\) and set \(\varphi(s)=Z(s)\cdot e\). Then
\[
 \varphi'(s)=Z'(s)\cdot e
 \ge(1-\eta)\abs{Z'(t)},
\]
so \(\varphi\) is strictly increasing. Moreover,
\[
 \abs{Z(s)}^2\ge\max\{r_0^2,\abs{\varphi(s)}^2\}
 \quad\text{on }\Omega.
\]
Changing variables from \(s\) to \(\varphi\) and enlarging the integration range to \(\R\),
\[
 \int_\Omega\frac{\dd s}{\abs{Z(s)}^2}
 \le\frac1{(1-\eta)\abs{Z'(t)}}
 \int_\R\frac{\dd y}{\max\{r_0^2,y^2\}}
 \le\frac{C_\eta}{\abs{Z'(t)}r_0}.
\]
\end{proof}

The next lemma replaces a source-dependent short interval by a prescribed longer interval.

\begin{lemma}[One-sided time extension]\label{lem:time-extension}
Let \(\alpha,\beta\in L^1((0,T))\) be nonnegative, and let
\(d\in C([0,T];[0,\infty))\) satisfy \(d(s)\le s\). Suppose that
\[
 \int_{s-d(s)}^s \alpha(u)\dd u
 \le\int_{s-d(s)}^s \beta(u)\dd u
 \quad\text{for every }s\in(0,T].
\]
Assume also that \(\alpha(s)=0\) for almost every \(s\) such that \(d(s)=0\).
Then, for every \(t\in(0,T]\) and every
\(\delta\in[d(t),t]\),
\[
 \int_{t-\delta}^t \alpha(s)\dd s
 \le2\int_{t-\delta}^t \beta(s)\dd s.
\]
\end{lemma}

This is Schaeffer's one-sided covering lemma
\cite{SchaefferAsymptotic}, in the precise form recorded in
\cite[Lemma~3.4]{ChenLi}. That lemma is stated for \(t<T\). To obtain the
endpoint when \(\delta>d(T)\), take \(t_n\uparrow T\) and
\(\delta_n=\delta-(T-t_n)\). For large \(n\), continuity of \(d\) gives
\(\delta_n\ge d(t_n)\), and passage to the limit yields the assertion.
The case \(\delta=d(T)<T\) follows by first approximating \(\delta\) from
above, while \(\delta=d(T)=T\) follows by applying the lemma on
\([0,t_n]\) and sending \(t_n\uparrow T\).

\subsection{The running envelope and record characteristics}

We now justify the record-time construction associated with \eqref{eq:envelope-def}.

\begin{lemma}[Continuity and attainment of the running energy]\label{lem:H-continuity}
The function \(\cH\) defined in \eqref{eq:envelope-def} is continuous and
nondecreasing. For every \(t\ge0\), there exist a latest time \(T\in[0,t]\)
and a point \((x_*,v_*)\in\supp f(T)\) such that
\begin{equation}
 \LE(T,x_*,v_*)=\cH(T)=\cH(t). \label{eq:record-attainment}
\end{equation}
\end{lemma}

\begin{proof}
Let \(K_0=\supp f_0\). By invariance of the support under the characteristic flow,
\[
 \cH(t)=
 \max_{\substack{0\le s\le t\\(x_0,v_0)\in K_0}}
 \LE\bigl(s,X(s;0,x_0,v_0),V(s;0,x_0,v_0)\bigr).
\]
The function under the maximum is jointly continuous in
\((s,x_0,v_0)\). Compactness gives attainment. Uniform continuity on bounded
time intervals gives the continuity of \(\cH\), while monotonicity is
immediate. The set of maximizing times is compact and nonempty, so it has a
latest element \(T\), which proves \eqref{eq:record-attainment}.
\end{proof}

The estimate \eqref{eq:chen-li-support} gives the a priori bound used to control lower-order terms.

\begin{lemma}[A priori growth of the labeled-energy envelope]\label{lem:H-baseline}
For every \(t\ge0\),
\begin{equation}
 \cH(t)^{1/2}
 \le C\tscale{t}^{2/15}\tlog{t}^{8/15}. \label{eq:H-baseline}
\end{equation}
Moreover,
\begin{equation}
 \abs{w(s,x,v)}\le\sqrt{2\cH(t)}
 \quad\text{whenever }0\le s\le t,
 \ (x,v)\in\supp f(s). \label{eq:w-by-H}
\end{equation}
\end{lemma}

\begin{proof}
By \eqref{eq:chen-li-support} and \cref{lem:U-decay},
\[
 \max_{(x,v)\in\supp f(s)}\LE(s,x,v)
 \le\frac12\theta(s)^2+\norm{U(s)}_\infty
 \le C\tscale{s}^{4/15}\tlog{s}^{16/15}+C.
\]
Taking the maximum over \(0\le s\le t\) proves
\eqref{eq:H-baseline}. The second assertion follows immediately from the definition of \(\cH\).
\end{proof}

\begin{lemma}[Comparability on a record characteristic]\label{lem:record-comparison}
There exist thresholds \(T_*\ge1\), \(Q_*\ge1\), constants
\(0<c_q<C_q\), and a fixed \(0<\kappa<1\) with the following property.
Let \(T\ge T_*\) be a record time, so that
\(\LE(T,x_*,v_*)=\cH(T)\) for some
\((x_*,v_*)\in\supp f(T)\). Let \((X_*,V_*)\) be the corresponding characteristic and set
\[
 Q:=\cH(T)^{1/2}.
\]
If \(Q\ge Q_*\), define
\begin{equation}
 \delta:=\min\left\{\frac T5,\cD(T,\kappa Q)\right\}. \label{eq:record-delta}
\end{equation}
Then
\begin{equation}
 c_qQ\le\abs{w_*(s)}\le C_qQ
 \quad\text{for every }s\in[T-\delta,T]. \label{eq:q-record-comparison}
\end{equation}
\end{lemma}

\begin{proof}
At the record time,
\[
 \frac12\abs{w_*(T)}^2
 =Q^2-U(T,x_*).
\]
By \eqref{eq:U-decay}, choosing \(Q_*\) sufficiently large gives
\begin{equation}
 cQ\le\abs{w_*(T)}\le CQ. \label{eq:q-at-T}
\end{equation}
Furthermore,
\[
 (\tscale{s}w_*(s))'=\tscale{s}E(s,X_*(s)),
\]
so for \(s\in[T-\delta,T]\),
\begin{equation}
 w_*(s)=\frac{1+T}{\tscale{s}}w_*(T)
 -\frac1{\tscale{s}}\int_s^T\tscale{u}E(u,X_*(u))\dd u. \label{eq:q-backward}
\end{equation}
Since \(\delta\le T/5\), one has \(\tscale{s}\asymp1+T\). By
\cref{lem:Delta-closed} and \(\delta\le\cD(T,\kappa Q)\),
\[
 \int_s^T\abs{E(u,X_*(u))}\dd u\le\kappa Q.
\]
Thus the second term in \eqref{eq:q-backward} is bounded by
\(C\kappa Q\). Fixing \(\kappa\) sufficiently small relative to the lower
constant in \eqref{eq:q-at-T}, and then fixing \(c_q,C_q\), proves the result.
\end{proof}

\section{Weighted estimates for the singular current}\label{sec:current}

The purpose of this section is to prove the weighted-current estimate stated in
\cref{prop:local-current}. Fix throughout this section a record time \(T\) satisfying the hypotheses of
\cref{lem:record-comparison}. Write
\[
 Q=\cH(T)^{1/2},
\]
and let \(\delta\) be defined by \eqref{eq:record-delta}. Since
\(\delta\le T/5\), for \(s\in[T-\delta,T]\) we have
\begin{equation}
 \tscale{s}\asymp\tscale{T},
 \qquad
 \tlog{s}\asymp\Tlog. \label{eq:record-comparable-time}
\end{equation}

\begin{proposition}[Weighted current estimate on a record interval]\label{prop:local-current}
At every sufficiently large record time with \(Q\ge Q_*\), and for every \(A>0\),
\begin{equation}
\begin{aligned}
 \int_{T-\delta}^{T}\abs{\cJ_*(s)}\dd s
 \le C\delta\Bigg[
 &\frac{Q}{\tscale{T}}+A\frac{Q}{\tscale{T}}
 +A\frac{\Tlog^2}{Q\tscale{T}}\\
 &+A\frac{\Tlog^{4/7}}{\tscale{T}^{5/7}}
 +\frac{Q}{A\tscale{T}}
 \Bigg].
\end{aligned}\label{eq:local-current}
\end{equation}
\end{proposition}

All the remaining subsections are devoted to the proof of
\cref{prop:local-current}. We first introduce the
\(\cG\)--\(\cB\)--\(\cU\) decomposition, then estimate the good region
\(\cG\), the five near-region branches of \(\cB\), and finally the crossing
region \(\cU\). The resulting bounds are assembled at the end of the section.

\subsection{The three-region mechanism}\label{subsec:three-regions}

By \eqref{eq:J-def}, the problem is reduced to estimating
\[
 \iiint_{[T-\delta,T]\times\R^3\times\R^3}
 \frac{\abs{w} f(s,x,v)}{\abs{x-X_*(s)}^2}
 \dd s\dd x\dd v.
\]
The decomposition below separates three mechanisms. The region \(\cG\)
contains low self-similar velocity, low relative velocity, or large spatial
separation; it is controlled directly by the dispersive density bounds and
mass conservation. The region \(\cB\) is spatially close to the target
trajectory, and the weight \(w\) is retained in its radial integration. The
remaining region \(\cU\) is controlled by the time during which two
characteristics can cross a small neighborhood of one another. The five
branches below encode the adaptive time scale of \cite{ChenLi}: they are five
realizations of the same near-region mechanism, not five separate ideas.

Choose the low-velocity cutoff
\begin{equation}
 P:=2\tscale{T}^{-1/2}. \label{eq:P-choice}
\end{equation}
Let \(A>0\) be a free parameter. For
\(s\in[T-\delta,T]\), set
\begin{equation*}
 z=x-X_*(s),
 \qquad
 a=v-V_*(s),
 \qquad
 w=v-\frac{x}{\tscale{s}}.
\end{equation*}

Fix \(0<\eps<1/20\) once and for all. Choose the constant \(c_d\) supplied
by \cref{lem:rescaling}, reducing it if necessary so that both
\eqref{eq:d-Delta} and \eqref{eq:d-global} hold with
\(C_w=\sqrt2\) and with the fixed \(\kappa\) from
\cref{lem:record-comparison}. Define the five local-time branches by
\begin{align*}
 d_1&=c_ds,
 &d_2&=c_d\frac{\tscale{s}\abs{a}^2}{\tlog{s}},
 &d_3&=c_d\frac{\tscale{s}^{5/7}\abs{a}}{\tlog{s}^{4/7}},\\
 d_4&=c_d\frac{\tscale{s}\abs{w}^2}{\tlog{s}},
 &d_5&=c_d\frac{\tscale{s}^{5/7}\abs{w}}{\tlog{s}^{4/7}}.
\end{align*}

With these branches, set
\[
 d(s,x,v):=\min_{1\le j\le5}d_j(s,x,v).
\]
When \(d>0\), define
\begin{equation*}
 r(s,x,v):=
 \frac{A}{\abs{a}\bigl(\tscale{s}^{-1}+\abs{w}^2\bigr)d(s,x,v)}.
\end{equation*}
When \(d=0\), set \(r=0\). The same formulas define \(d\) and \(r\) for every
\(0\le s\le T\); this extension is used only in the time-covering argument.

On the record cylinder \([T-\delta,T]\times\R^3\times\R^3\), set
\begin{align*}
 \cG:=\biggl\{(s,x,v):\;&\abs{w}\le P
 \ \text{or}\ \abs{a}\le P
 \ \text{or}\ \abs{z}\ge\frac12\tscale{s}P\biggr\},\\
 \cB&:=\bigl\{(s,x,v):\abs{z}\le r(s,x,v)\bigr\}\setminus\cG,\\
 \cU&:=\bigl([T-\delta,T]\times\R^3\times\R^3\bigr)
 \setminus(\cG\cup\cB).
\end{align*}
The near set is divided into five disjoint pieces: a point of \(\cB\) is
assigned to \(\cB_j\) if \(j\) is the smallest index for which
\(d=d_j\). We extend \(\one_{\cU}\) by zero outside the record interval.
Here and below, a triple integral over a subset of the record cylinder is
understood with respect to \(\dd s\dd x\dd v\).

The identity
\begin{equation}
 w_*(s)=w-a+\frac z{\tscale{s}} \label{eq:q-w-a-z}
\end{equation}
provides the separation needed in all near-region estimates.

\begin{lemma}[Separation on the near region]\label{lem:B-separation}
After increasing \(T_*\) and \(Q_*\), there exist constants
\(0<c_B<C_B\) such that on \(\cB\),
\begin{equation}
 \abs{a}+\abs{w}\ge c_BQ,
 \qquad
 P<\abs{a},\abs{w}\le C_BQ. \label{eq:B-separation}
\end{equation}
\end{lemma}

\begin{proof}
Since \(\cB\subset\cG^c\),
\[
 \frac{\abs{z}}{\tscale{s}}<\frac P2,
 \qquad
 \abs{a}>P,
 \qquad
 \abs{w}>P.
\]
Combining \eqref{eq:q-w-a-z} with
\eqref{eq:q-record-comparison},
\[
 \abs{a}+\abs{w}
 \ge\abs{w_*(s)}-\frac{\abs{z}}{\tscale{s}}
 \ge c_qQ-\frac P2.
\]
For large \(T\) and \(Q\ge Q_*\), this gives the first inequality. By
\eqref{eq:w-by-H}, \(\abs{w}\le\sqrt2Q\). Finally,
\[
 \abs{a}\le\abs{w}+\abs{w_*(s)}+\frac{\abs{z}}{\tscale{s}}
 \le C_BQ.
\]
\end{proof}

The dependence of one velocity variable on \(z\) in the other coordinate
system is the main bookkeeping issue in the radial estimates. We isolate it
now.

\begin{lemma}[Frozen complementary velocity]\label{lem:freezing}
Fix \(s\in[T-\delta,T]\). In \((z,a)\) coordinates define
\begin{equation*}
 \tw_s(a):=w_*(s)+a,
 \qquad
 w=\tw_s(a)-\frac z{\tscale{s}}.
\end{equation*}
In \((z,w)\) coordinates define
\begin{equation*}
 \ta_s(w):=w-w_*(s),
 \qquad
 a=\ta_s(w)+\frac z{\tscale{s}}.
\end{equation*}
On \(\cG^c\),
\begin{align}
 \frac12\abs{w}\le\abs{\tw_s(a)}\le\frac32\abs{w},
 &\qquad
 \tscale{s}^{-1}+\abs{\tw_s(a)}^2
 \asymp\tscale{s}^{-1}+\abs{w}^2, \label{eq:freeze-w-comparison}\\
 \frac12\abs{a}\le\abs{\ta_s(w)}\le\frac32\abs{a}.
 & \label{eq:freeze-a-comparison}
\end{align}
Moreover, suppose \(D\subset\cG^c\) is measurable and, in \((z,a)\) coordinates,
\(\abs{z}\le R_s(a)\) on \(D\), where
\(R_s:\R^3\to[0,\infty]\) is measurable. Then
\begin{equation}
 \iint_D\frac{\abs{w} f(s,x,v)}{\abs{z}^2}\dd z\dd a
 \le C\int_{\R^3}\abs{\tw_s(a)}R_s(a)\dd a. \label{eq:frozen-radial-a}
\end{equation}
If, in \((z,w)\) coordinates, \(\abs{z}\le R_s(w)\) on \(D\) for a
nonnegative measurable \(R_s\), then
\begin{equation}
 \iint_D\frac{\abs{w} f(s,x,v)}{\abs{z}^2}\dd z\dd w
 \le C\int_{\R^3}\abs{w} R_s(w)\dd w. \label{eq:frozen-radial-w}
\end{equation}
\end{lemma}

\begin{proof}
On \(\cG^c\), \(\abs{z}/\tscale{s}<P/2\), while
\(\abs{a}>P\) and \(\abs{w}>P\). Hence
\[
 \abs{\tw_s(a)-w}=\frac{\abs{z}}{\tscale{s}}<\frac12\abs{w},
\]
which proves the first comparison in \eqref{eq:freeze-w-comparison}; the
quadratic comparison follows immediately. The proof of
\eqref{eq:freeze-a-comparison} is identical.

Both affine maps
\((x,v)\mapsto(z,a)\) and
\((x,v)\mapsto(z,w)\) have unit Jacobian. Since
\(\norm{f(s)}_\infty\le\norm{f_0}_\infty\),
\[
 \int_{\abs{z}\le R}\frac{\dd z}{\abs{z}^2}=4\pi R.
\]
Using \(\abs{w}\asymp\abs{\tw_s(a)}\) on \(\cG^c\) gives
\eqref{eq:frozen-radial-a}; \eqref{eq:frozen-radial-w} is immediate.
\end{proof}

By \eqref{eq:J-def},
\begin{equation}
 \abs{\cJ_*(s)}
 \le C\iint_{\R^3\times\R^3}
 \frac{\abs{w} f(s,x,v)}{\abs{x-X_*(s)}^2}\dd x\dd v. \label{eq:J-upper}
\end{equation}
We estimate the three pieces of the decomposition separately.

\subsection{The good region}

\begin{lemma}[Contribution of \(\cG\)]\label{lem:G}
One has
\begin{equation*}
 \iiint_{\cG}\frac{\abs{w} f}{\abs{z}^2}
 \le C\delta\frac{Q}{\tscale{T}}.
\end{equation*}
\end{lemma}

\begin{proof}
For fixed \(s\), define
\[
 \bar\rho(s,x):=
 \int_{\{\abs{w}\le P\}\cup\{\abs{a}\le P\}}
 f(s,x,v)\dd v.
\]
At fixed \(x\), both velocity conditions describe balls of radius \(P\). Thus
\[
 \norm{\bar\rho(s)}_{L^\infty}\le CP^3,
 \qquad
 \norm{\bar\rho(s)}_{L^{5/3}}
 \le\norm{\rho(s)}_{L^{5/3}}.
\]
Applying \cref{lem:coulomb-interpolation} with \(p=5/3\) and using
\eqref{eq:rho53},
\begin{equation}
 \sup_y\int_{\R^3}
 \frac{\bar\rho(s,x)}{\abs{x-y}^2}\dd x
 \le C\tscale{s}^{-1/3}P^{4/3}. \label{eq:G-near}
\end{equation}
On the remaining part of \(\cG\),
\(\abs{z}\ge\tscale{s}P/2\), and mass conservation gives
\begin{equation}
 \iint_{\abs{z}\ge\tscale{s}P/2}
 \frac{f(s,x,v)}{\abs{z}^2}\dd x\dd v
 \le C\tscale{s}^{-2}P^{-2}. \label{eq:G-far}
\end{equation}
By \eqref{eq:w-by-H}, \(\abs{w}\le\sqrt2Q\) throughout the record interval. Combining
\eqref{eq:G-near}--\eqref{eq:G-far},
\[
 \iiint_{\cG}\frac{\abs{w} f}{\abs{z}^2}
 \le C\delta Q
 \left(\tscale{T}^{-1/3}P^{4/3}+\tscale{T}^{-2}P^{-2}\right).
\]
With \(P=2\tscale{T}^{-1/2}\), both terms in parentheses are comparable to \(\tscale{T}^{-1}\), proving the claim.
\end{proof}

\subsection{The five near-region branches}

We estimate the five branches in turn. The \(a\)-branches and the
\(w\)-branches occur in parallel; the decisive point is to retain the factor
\(\abs{w}\) in the two \(w\)-branch integrations.

\begin{lemma}[Contribution of \(\cB_1\)]\label{lem:B1}
\begin{equation}
 \iiint_{\cB_1}\frac{\abs{w} f}{\abs{z}^2}
 \le CA\delta\frac{Q}{\tscale{T}}. \label{eq:B1}
\end{equation}
\end{lemma}

\begin{proof}
On \(\cB_1\),
\begin{equation}
 r\le\frac{CA}{s\abs{a}(\tscale{s}^{-1}+\abs{w}^2)}. \label{eq:B1-radius}
\end{equation}
Split \(\cB_1\) into the regions \(\abs{w}\le\abs{a}\) and
\(\abs{a}<\abs{w}\).

On the first region, use \((z,w)\) coordinates. Since
\(\abs{a}\ge\abs{w}\), \eqref{eq:B1-radius} implies
\[
 r\le\frac{CA}{s\abs{w}(\tscale{s}^{-1}+\abs{w}^2)}.
\]
Therefore, by \eqref{eq:frozen-radial-w},
\begin{align*}
 \iint_{\cB_1\cap\{\abs{w}\le\abs{a}\}}
 \frac{\abs{w} f}{\abs{z}^2}\dd x\dd v
 &\le\frac{CA}{s}
 \int_{P<\abs{w}\le C_BQ}
 \frac{\dd w}{\tscale{s}^{-1}+\abs{w}^2}\\
 &\le\frac{CAQ}{s}.
\end{align*}
Indeed, in radial variables the last integral is
\(C\int_P^{C_BQ}r^2(\tscale{s}^{-1}+r^2)^{-1}\dd r\le CQ\).

On the second region, use \((z,a)\) coordinates. By
\cref{lem:freezing}, \(\abs{\tw_s(a)}\asymp\abs{w}\), and the inequality
\(\abs{a}<\abs{w}\) gives
\(\abs{\tw_s(a)}\ge c\abs{a}\). Hence
\[
 r\le\frac{CA}{s\abs{a}(\tscale{s}^{-1}+\abs{\tw_s(a)}^2)}.
\]
By \eqref{eq:frozen-radial-a},
\begin{align*}
 \iint_{\cB_1\cap\{\abs{a}<\abs{w}\}}
 \frac{\abs{w} f}{\abs{z}^2}\dd x\dd v
 &\le\frac{CA}{s}
 \int_{P<\abs{a}\le C_BQ}
 \frac{\abs{\tw_s(a)}}
 {\abs{a}(\tscale{s}^{-1}+\abs{\tw_s(a)}^2)}\dd a\\
 &\le\frac{CA}{s}
 \int_{P<\abs{a}\le C_BQ}\frac{\dd a}{\abs{a}^2}
 \le\frac{CAQ}{s}.
\end{align*}
Since \(s\asymp\tscale{T}\) on the record interval, integration in time proves \eqref{eq:B1}.
\end{proof}

\begin{lemma}[Contribution of \(\cB_2\)]\label{lem:B2}
\begin{equation}
 \iiint_{\cB_2}\frac{\abs{w} f}{\abs{z}^2}
 \le CA\delta\frac{\Tlog^2}{Q\tscale{T}}. \label{eq:B2}
\end{equation}
\end{lemma}

\begin{proof}
Since \(d_2\le d_4\) on \(\cB_2\),
\(\abs{a}\le\abs{w}\). Together with
\eqref{eq:B-separation}, this gives \(\abs{w}\ge cQ\). In
\((z,a)\) coordinates, \cref{lem:freezing} yields
\[
 r\le\frac{CA\tlog{s}}
 {\tscale{s}\abs{a}^3(\tscale{s}^{-1}+\abs{\tw_s(a)}^2)},
 \qquad
 \abs{\tw_s(a)}\ge cQ.
\]
Therefore,
\begin{align*}
 \iint_{\cB_2}\frac{\abs{w} f}{\abs{z}^2}\dd x\dd v
 &\le\frac{CA\tlog{s}}{\tscale{s}}
 \int_{P<\abs{a}\le C_BQ}
 \frac{\abs{\tw_s(a)}}
 {\abs{a}^3(\tscale{s}^{-1}+\abs{\tw_s(a)}^2)}\dd a\\
 &\le\frac{CA\tlog{s}}{\tscale{s}Q}
 \int_{P<\abs{a}\le C_BQ}\frac{\dd a}{\abs{a}^3}\\
 &\le\frac{CA\tlog{s}}{\tscale{s}Q}\log\frac{C_BQ}{P}.
\end{align*}
By \eqref{eq:H-baseline} and \eqref{eq:P-choice},
\[
 \log\frac{C_BQ}{P}\le C\Tlog.
\]
Time integration proves \eqref{eq:B2}.
\end{proof}

\begin{lemma}[Contribution of \(\cB_3\)]\label{lem:B3}
\begin{equation}
 \iiint_{\cB_3}\frac{\abs{w} f}{\abs{z}^2}
 \le CA\delta\frac{\Tlog^{4/7}}{\tscale{T}^{5/7}}. \label{eq:B3}
\end{equation}
\end{lemma}

\begin{proof}
On \(\cB_3\), \(d_3\le d_5\), so
\(\abs{a}\le\abs{w}\). Hence \(\abs{w}\ge cQ\) by
\eqref{eq:B-separation}. In \((z,a)\) coordinates,
\[
 r\le\frac{CA\tlog{s}^{4/7}}
 {\tscale{s}^{5/7}\abs{a}^2
 (\tscale{s}^{-1}+\abs{\tw_s(a)}^2)}.
\]
Choose a fixed constant \(c_1\in(0,c_B/4)\) and split at
\(\abs{a}=c_1Q\). After increasing \(Q_*\), we may assume
\(P<c_1Q\).

If \(\abs{a}\ge c_1Q\), then
\(\abs{\tw_s(a)}\asymp\abs{w}\) and
\(\abs{\tw_s(a)}\ge c\abs{a}\). Therefore,
\begin{align*}
 \iint_{\cB_3\cap\{\abs{a}\ge c_1Q\}}
 \frac{\abs{w} f}{\abs{z}^2}\dd x\dd v
 &\le\frac{CA\tlog{s}^{4/7}}{\tscale{s}^{5/7}}
 \int_{c_1Q\le\abs{a}\le C_BQ}
 \frac{\dd a}{\abs{a}^3}\\
 &\le\frac{CA\tlog{s}^{4/7}}{\tscale{s}^{5/7}}.
\end{align*}
If \(\abs{a}<c_1Q\), separation implies
\(\abs{\tw_s(a)}\asymp\abs{w}\ge cQ\). Hence
\begin{align*}
 \iint_{\cB_3\cap\{\abs{a}<c_1Q\}}
 \frac{\abs{w} f}{\abs{z}^2}\dd x\dd v
 &\le\frac{CA\tlog{s}^{4/7}}{\tscale{s}^{5/7}Q}
 \int_{P<\abs{a}<c_1Q}\frac{\dd a}{\abs{a}^2}\\
 &\le\frac{CA\tlog{s}^{4/7}}{\tscale{s}^{5/7}}.
\end{align*}
Integrating in time proves \eqref{eq:B3}.
\end{proof}

\begin{lemma}[Contribution of the quadratic \(w\)-branch]\label{lem:B4}
\begin{equation}
 \iiint_{\cB_4}\frac{\abs{w} f}{\abs{z}^2}
 \le CA\delta\frac{\Tlog^2}{Q\tscale{T}}. \label{eq:B4}
\end{equation}
\end{lemma}

\begin{proof}
The inequality \(d_4\le d_2\) gives
\(\abs{w}\le\abs{a}\). By separation, \(\abs{a}\ge cQ\). Thus, in
\((z,w)\) coordinates,
\[
 r\le\frac{CA\tlog{s}}
 {\tscale{s}Q\abs{w}^2(\tscale{s}^{-1}+\abs{w}^2)}.
\]
Using \eqref{eq:frozen-radial-w},
\begin{align*}
 \iint_{\cB_4}\frac{\abs{w} f}{\abs{z}^2}\dd x\dd v
 &\le\frac{CA\tlog{s}}{\tscale{s}Q}
 \int_{P<\abs{w}\le C_BQ}
 \frac{\dd w}{\abs{w}(\tscale{s}^{-1}+\abs{w}^2)}\\
 &\le\frac{CA\tlog{s}}{\tscale{s}Q}
 \int_P^{C_BQ}\frac{r}{\tscale{s}^{-1}+r^2}\dd r\\
 &\le\frac{CA\tlog{s}}{\tscale{s}Q}
 \log\left(
 \frac{\tscale{s}^{-1}+C_B^2Q^2}
 {\tscale{s}^{-1}+P^2}\right)
 \le\frac{CA\tlog{s}^2}{\tscale{s}Q}.
\end{align*}
The last step follows from \eqref{eq:H-baseline} and \eqref{eq:P-choice}. Time integration proves \eqref{eq:B4}.
\end{proof}

\begin{lemma}[Contribution of the linear \(w\)-branch]\label{lem:B5}
\begin{equation}
 \iiint_{\cB_5}\frac{\abs{w} f}{\abs{z}^2}
 \le CA\delta\frac{\Tlog^{4/7}}{\tscale{T}^{5/7}}. \label{eq:B5}
\end{equation}
\end{lemma}

\begin{proof}
Here \(d_5\le d_3\), so \(\abs{w}\le\abs{a}\), and
\eqref{eq:B-separation} gives \(\abs{a}\ge cQ\). In \((z,w)\) coordinates,
\[
 r\le\frac{CA\tlog{s}^{4/7}}
 {\tscale{s}^{5/7}Q\abs{w}(\tscale{s}^{-1}+\abs{w}^2)}.
\]
Therefore,
\begin{align*}
 \iint_{\cB_5}\frac{\abs{w} f}{\abs{z}^2}\dd x\dd v
 &\le\frac{CA\tlog{s}^{4/7}}{\tscale{s}^{5/7}Q}
 \int_{P<\abs{w}\le C_BQ}
 \frac{\dd w}{\tscale{s}^{-1}+\abs{w}^2}\\
 &\le\frac{CA\tlog{s}^{4/7}}{\tscale{s}^{5/7}Q}
 \int_0^{C_BQ}\frac{r^2}{\tscale{s}^{-1}+r^2}\dd r\\
 &\le\frac{CA\tlog{s}^{4/7}}{\tscale{s}^{5/7}}.
\end{align*}
The factor \(\abs{w}\) in the current is essential here: it cancels the factor
\(\abs{w}^{-1}\) in the branch radius before the radial integration. Integrating in time proves \eqref{eq:B5}.
\end{proof}

Combining \cref{lem:B1,lem:B2,lem:B3,lem:B4,lem:B5},
\begin{equation}
 \iiint_{\cB}\frac{\abs{w} f}{\abs{z}^2}
 \le C\delta\left(
 A\frac{Q}{\tscale{T}}
 +A\frac{\Tlog^2}{Q\tscale{T}}
 +A\frac{\Tlog^{4/7}}{\tscale{T}^{5/7}}
 \right). \label{eq:B-total}
\end{equation}

\subsection{The crossing region}

Let \((X(s),V(s))\) be a source characteristic issued from the support, so
that \((X(s),V(s))\in\supp f(s)\). Put
\begin{equation*}
 Z(s)=X(s)-X_*(s),
 \qquad
 a(s)=V(s)-V_*(s),
 \qquad
 W(s)=V(s)-\frac{X(s)}{\tscale{s}}.
\end{equation*}
At a terminal time \(t\), write
\[
 d_t=d(t,X(t),V(t)),
 \qquad
 r_t=r(t,X(t),V(t)),
 \qquad
 a_t=a(t),
 \qquad
 W_t=W(t).
\]

The next lemma records the local comparison statements used in the crossing
argument.

\begin{lemma}[Stability on the source-dependent interval]\label{lem:source-stability}
Assume \(d_t>0\). For every \(s\in[t-d_t,t]\),
\begin{equation}
 \frac12\abs{a_t}\le\abs{a(s)}\le\frac32\abs{a_t}, \label{eq:a-stability}
\end{equation}
and
\begin{equation}
 \frac12\abs{W_t}\le\abs{W(s)}\le\frac32\abs{W_t}. \label{eq:W-stability}
\end{equation}
Consequently,
\begin{equation}
 d(s,X(s),V(s))\asymp d_t,
 \qquad
 r(s,X(s),V(s))\asymp r_t. \label{eq:d-r-stability}
\end{equation}
All comparison constants are independent of the source characteristic and of \(t\).
\end{lemma}

\begin{proof}
By \cref{lem:rescaling},
\[
 d_t\le\cD(t,\eps\abs{a_t}),
 \qquad
 d_t\le\cD(t,\eps\abs{W_t}),
 \qquad
 d_t\le t/5.
\]
Both the source and target curves end in \(\supp f(t)\). Therefore, by
\cref{lem:Delta-closed},
\begin{align*}
 \abs{a(s)-a_t}
 &\le\int_s^t\abs{E(u,X(u))}\dd u
 +\int_s^t\abs{E(u,X_*(u))}\dd u\\
 &\le2\eps\abs{a_t}.
\end{align*}
Since \(2\eps<1/10\), this proves \eqref{eq:a-stability}.

Next,
\[
 (\tscale{s}W(s))'=\tscale{s}E(s,X(s)),
\]
so
\begin{equation}
 W(s)=\frac{\tscale{t}}{\tscale{s}}W_t
 -\frac1{\tscale{s}}\int_s^t\tscale{u}E(u,X(u))\dd u. \label{eq:W-backward}
\end{equation}
Because \(t-s\le d_t\le t/5\),
\[
 1\le\frac{\tscale{t}}{\tscale{s}}\le\frac54
\]
for every \(t>0\). Furthermore,
\[
 \frac1{\tscale{s}}\int_s^t\tscale{u}\abs{E(u,X(u))}\dd u
 \le\frac{\tscale{t}}{\tscale{s}}\eps\abs{W_t}
 \le\frac54\eps\abs{W_t}.
\]
Subtracting \(W_t\) in \eqref{eq:W-backward},
\[
 \abs{W(s)-W_t}
 \le\frac{t-s}{\tscale{s}}\abs{W_t}
 +\frac54\eps\abs{W_t}
 \le\left(\frac14+\frac54\eps\right)\abs{W_t}.
\]
Since \(\eps<1/20\), the last coefficient is less than \(1/2\), which proves
\eqref{eq:W-stability}.

The ratios \(s/t\), \(\tscale{s}/\tscale{t}\), and
\(\tlog{s}/\tlog{t}\) are bounded above and below on the local interval.
Together with \eqref{eq:a-stability}--\eqref{eq:W-stability}, this makes each
of the five branches defining \(d(s,X(s),V(s))\) comparable with its
terminal value. Taking the minimum proves the first comparison in
\eqref{eq:d-r-stability}; the second then follows directly from the definition of \(r\).
\end{proof}

\begin{lemma}[Local source crossing]\label{lem:local-crossing}
Whenever \(d_t>0\),
\begin{equation}
 \int_{t-d_t}^{t}
 \frac{\one_{\cU}(s,X(s),V(s))}{\abs{Z(s)}^2}\dd s
 \le\frac CA\int_{t-d_t}^{t}
 \left(\tscale{s}^{-1}+\abs{W(s)}^2\right)\dd s. \label{eq:local-crossing}
\end{equation}
\end{lemma}

\begin{proof}
On \(\cU\),
\[
 \abs{Z(s)}>r(s,X(s),V(s)).
\]
By \cref{lem:source-stability}, this is bounded below by
\(c_rr_t\) for a fixed \(c_r>0\). Moreover,
\eqref{eq:a-stability} gives
\[
 \abs{a(s)-a_t}\le2\eps\abs{a_t}
\]
with \(2\eps<1/2\). Applying \cref{lem:geometric-crossing} to
\(Z'=a\),
\[
 \int_{t-d_t}^{t}
 \frac{\one_{\cU}(s,X(s),V(s))}{\abs{Z(s)}^2}\dd s
 \le\frac{C}{\abs{a_t}r_t}.
\]
By the definition of \(r_t\),
\[
 \frac1{\abs{a_t}r_t}
 =\frac{d_t}{A}
 \left(\tscale{t}^{-1}+\abs{W_t}^2\right).
\]
The last factor is comparable throughout \([t-d_t,t]\) by
\cref{lem:source-stability}, which proves \eqref{eq:local-crossing}.
\end{proof}

\begin{lemma}[Contribution of \(\cU\)]\label{lem:U}
One has
\begin{equation}
 \iiint_{\cU}\frac{\abs{w} f}{\abs{z}^2}
 \le C\delta\frac{Q}{A\tscale{T}}. \label{eq:U-estimate}
\end{equation}
\end{lemma}

\begin{proof}
Fix a source characteristic ending at \((x,v)\in\supp f(T)\). Define
\[
 \alpha(s):=
 \frac{\one_{\cU}(s,X(s),V(s))}
 {\abs{X(s)-X_*(s)}^2},
 \qquad
 \beta(s):=\frac CA
 \left(\tscale{s}^{-1}+\abs{W(s)}^2\right),
\]
and
\[
 d(s):=d(s,X(s),V(s)).
\]
As usual, the quotient defining \(\alpha\) is extended by zero off
\(\cU\); in particular, it is assigned the value zero at collision times
that do not belong to \(\cU\).
The local estimate required in \cref{lem:time-extension} is
\eqref{eq:local-crossing}. If \(d(s)=0\), then either \(s=0\),
\(a(s)=0\), or \(W(s)=0\). On the record interval the last two alternatives
place the point in \(\cG\), while outside the record interval
\(\one_{\cU}=0\). Hence \(\alpha(s)=0\) whenever \(d(s)=0\).

At time \(T\), \eqref{eq:w-by-H} gives
\(\abs{W(T)}\le\sqrt2Q\). The choice of \(c_d\) fixed above uses
\(C_w=\sqrt2\) and the same \(\kappa\) as in
\eqref{eq:record-delta}. Hence
\[
 d(T,x,v)\le\min\left\{\frac T5,\cD(T,\kappa Q)\right\}=\delta.
\]
Strictly speaking, \(\alpha\) may be unbounded before the crossing estimate
is used. Apply \cref{lem:time-extension} first to
\(\alpha_N=\min\{\alpha,N\}\); the local inequality remains valid because
\(\alpha_N\le\alpha\). Letting \(N\to\infty\) by monotone convergence, we obtain
\begin{equation*}
 \int_{T-\delta}^{T}
 \frac{\one_{\cU}(s,X(s),V(s))}
 {\abs{X(s)-X_*(s)}^2}\dd s
 \le\frac CA\int_{T-\delta}^{T}
 \left(\tscale{s}^{-1}+\abs{W(s)}^2\right)\dd s.
\end{equation*}

Using phase-space measure preservation and
\eqref{eq:f-characteristics}, and recalling that
\(\abs{W}\le\sqrt2Q\) on the record interval,
\begin{align*}
 \iiint_{\cU}\frac{\abs{w} f}{\abs{z}^2}
 &\le CQ\iint f(T,x,v)
 \int_{T-\delta}^{T}
 \frac{\one_{\cU}(s,X(s),V(s))}
 {\abs{X(s)-X_*(s)}^2}\dd s\dd x\dd v\\
 &\le\frac{CQ}{A}
 \int_{T-\delta}^{T}\iint
 \left(\tscale{s}^{-1}+\abs{w(s,x,v)}^2\right)
 f(s,x,v)\dd x\dd v\dd s.
\end{align*}
Mass conservation and \eqref{eq:w2} imply that the inner integral is bounded by
\(C\tscale{s}^{-1}\). Using \eqref{eq:record-comparable-time} proves
\eqref{eq:U-estimate}.
\end{proof}

\begin{proof}[Proof of \cref{prop:local-current}]
Combine \eqref{eq:J-upper}, \cref{lem:G},
\eqref{eq:B-total}, and \cref{lem:U}; this gives
\eqref{eq:local-current} and completes the proof.
\end{proof}

\section{Closing the labeled energy and long-time consequences}\label{sec:conclusion}

The purpose of this section is to convert the local current estimate of
\cref{prop:local-current} into a global bound for the labeled-energy envelope.
The central conclusion is the following proposition.

\begin{proposition}[Growth of the labeled-energy envelope]\label{prop:H-growth}
For every \(t\ge0\),
\begin{equation}
 \cH(t)\le C\tscale{t}^{4/21}\tlog{t}^{8/21}. \label{eq:H-growth}
\end{equation}
\end{proposition}

Once \cref{prop:H-growth} is known, the velocity-support and field estimates in
\cref{thm:main} follow directly. The remainder of the section proves this
proposition in two steps: we first optimize the free geometric parameter in
\cref{prop:local-current}, and then iterate the resulting estimate over
those record intervals. The short proof of \cref{thm:main} is given at the
end.

\subsection{Optimization of the geometric parameter}

\begin{proposition}[Optimized record-time estimate]\label{prop:optimized-current}
At every sufficiently large record time with \(Q\ge Q_*\), with \(\delta\)
defined by \eqref{eq:record-delta},
\begin{equation*}
 \int_{T-\delta}^{T}\abs{\cJ_*(s)}\dd s
 \le C\delta\cH(T)^{1/4}\tscale{T}^{-6/7}\Tlog^{2/7}.
\end{equation*}
\end{proposition}

\begin{proof}
Choose
\[
 A:=Q^{1/2}\tscale{T}^{-1/7}\Tlog^{-2/7}.
\]
Then the two dominant terms balance:
\begin{equation}
 A\frac{\Tlog^{4/7}}{\tscale{T}^{5/7}}
 =\frac{Q}{A\tscale{T}}
 =Q^{1/2}\tscale{T}^{-6/7}\Tlog^{2/7}. \label{eq:dominant-balance}
\end{equation}
Divide the remaining three terms in \eqref{eq:local-current} by the quantity
on the right-hand side of \eqref{eq:dominant-balance}. The ratios are
\begin{equation}\label{eq:lower-order-ratios}
\begin{aligned}
 Q^{1/2}\tscale{T}^{-1/7}\Tlog^{-2/7},
 \qquad&
 Q\tscale{T}^{-2/7}\Tlog^{-4/7},
 \\
 &Q^{-1}\tscale{T}^{-2/7}\Tlog^{10/7}.
\end{aligned}
\end{equation}
By \eqref{eq:H-baseline},
\[
 Q\le C\tscale{T}^{2/15}\Tlog^{8/15}.
\]
Hence the first two ratios in \eqref{eq:lower-order-ratios} are bounded, and in fact tend to zero:
\[
 Q^{1/2}\tscale{T}^{-1/7}\Tlog^{-2/7}
 \le C\tscale{T}^{-8/105}\Tlog^{-2/105},
\]
\[
 Q\tscale{T}^{-2/7}\Tlog^{-4/7}
 \le C\tscale{T}^{-16/105}\Tlog^{-4/105}.
\]
The third ratio is bounded by
\(Q_*^{-1}\tscale{T}^{-2/7}\Tlog^{10/7}\), which also tends to zero. Since
\(Q^{1/2}=\cH(T)^{1/4}\), the result follows.
\end{proof}

\begin{proof}[Proof of \cref{prop:H-growth}]

We use the concavity inequality
\begin{equation}
 x^{3/4}-y^{3/4}
 \le (x-y)x^{-1/4},
 \qquad x\ge y\ge0,
 \quad x>0. \label{eq:concavity}
\end{equation}

Fix \(t\ge0\). By \cref{lem:H-continuity}, choose the latest record time
\(T_0\le t\) and \((x_0,v_0)\in\supp f(T_0)\) such that
\[
 \LE(T_0,x_0,v_0)=\cH(T_0)=\cH(t).
\]
If either \(T_0<T_*\) or \(\cH(T_0)<Q_*^2\), then
\(\cH(t)\) is bounded by a constant depending only on the initial datum, and
the conclusion follows. We therefore assume that both thresholds are
exceeded.

Let \((X_*,V_*)\) be the characteristic passing through
\((x_0,v_0)\) at time \(T_0\), and let \(\delta_0\) be given by
\eqref{eq:record-delta}. Integrating
\eqref{eq:labeled-inequality} along this characteristic and dropping the nonnegative damping term,
\[
 \cH(T_0)=\LE_*(T_0)
 \le\LE_*(T_0-\delta_0)
 +\int_{T_0-\delta_0}^{T_0}\abs{\cJ_*(s)}\dd s.
\]
The first term is bounded by
\[
 \LE_*(T_0-\delta_0)
 \le\cH(T_0-\delta_0).
\]
By \cref{prop:optimized-current},
\begin{equation*}
 \cH(T_0)
 \le\cH(T_0-\delta_0)
 +C\delta_0\cH(T_0)^{1/4}
 \tscale{T_0}^{-6/7}\tlog{T_0}^{2/7}.
\end{equation*}

Proceed recursively. Once \(T_n\) is chosen, set
\(S_n=T_n-\delta_n\), and let
\(T_{n+1}\le S_n\) be the latest record time given by
\cref{lem:H-continuity}; thus, for some
\((x_{n+1},v_{n+1})\in\supp f(T_{n+1})\),
\[
 \LE(T_{n+1},x_{n+1},v_{n+1})
 =\cH(T_{n+1})=\cH(S_n).
\]
Put
\[
 M_n:=\cH(T_n),
 \qquad
 g(s):=\tscale{s}^{-6/7}\tlog{s}^{2/7}.
\]
As long as \(T_n\ge T_*\) and \(M_n\ge Q_*^2\), the preceding argument gives
\begin{equation}
 M_n\le M_{n+1}+C\delta_nM_n^{1/4}g(T_n). \label{eq:record-recursion}
\end{equation}
The interiors of record intervals
\([T_n-\delta_n,T_n]\) are pairwise disjoint because
\(T_{n+1}\le T_n-\delta_n\).

We also verify that the construction terminates after finitely many steps. From
\eqref{eq:Delta-lower}, for \(T_n\ge T_*\) and
\(M_n^{1/2}\ge Q_*\),
\begin{align*}
 \delta_n
 &\ge c\min\left\{
 T_n,
 \frac{M_n\tscale{T_n}}{\tlog{T_n}},
 \frac{M_n^{1/2}\tscale{T_n}^{5/7}}
 {\tlog{T_n}^{4/7}}
 \right\}\\
 &\ge c_\delta>0,
\end{align*}
where \(c_\delta\) depends only on the fixed thresholds. Since these intervals have pairwise disjoint interiors and are contained in $[0,t]$, only finitely many such
steps are possible before either the time or the energy enters the bounded
regime.

Since \(T_{n+1}\le T_n\) and \(\cH\) is nondecreasing, one has
\(M_{n+1}\le M_n\). We may therefore apply \eqref{eq:concavity} to
\eqref{eq:record-recursion}:
\begin{equation}
 M_n^{3/4}-M_{n+1}^{3/4}
 \le C\delta_ng(T_n). \label{eq:concave-step}
\end{equation}
Because \(\delta_n\le T_n/5\), one has
\(g(s)\asymp g(T_n)\) on
\([T_n-\delta_n,T_n]\). Summing
\eqref{eq:concave-step} over these record intervals and absorbing the bounded terminal value gives
\begin{align*}
 \cH(t)^{3/4}
 &\le C+C\sum_n\delta_ng(T_n)\\
 &\le C+C\int_1^t\tscale{s}^{-6/7}\tlog{s}^{2/7}\dd s\\
 &\le C\tscale{t}^{1/7}\tlog{t}^{2/7}.
\end{align*}
Raising to the power \(4/3\) proves \eqref{eq:H-growth}.
\end{proof}

\subsection{Velocity support and field decay}

\begin{proof}[Proof of \cref{thm:main}]
By the definitions of \(\LE\) and \(\cH\),
\[
 \frac12\theta(t)^2
 \le\max_{(x,v)\in\supp f(t)}\LE(t,x,v)
 \le\cH(t).
\]
Using \cref{prop:H-growth},
\begin{equation}
 \theta(t)\le C\tscale{t}^{2/21}\tlog{t}^{4/21}. \label{eq:theta-final}
\end{equation}
Along every characteristic,
\begin{equation*}
 \frac{\dd}{\dd t}\left(\frac{X(t)}{\tscale{t}}\right)
 =\frac1{\tscale{t}}
 \left(V(t)-\frac{X(t)}{\tscale{t}}\right).
\end{equation*}
The compact support of the initial datum and
\eqref{eq:theta-final} therefore imply
\begin{align*}
 \frac{\abs{X(t)}}{\tscale{t}}
 &\le C+C\int_0^t
 \tscale{s}^{-1+2/21}\tlog{s}^{4/21}\dd s\\
 &\le C\tscale{t}^{2/21}\tlog{t}^{4/21}.
\end{align*}
Consequently,
\[
 \abs{V(t)}
 \le\abs{V(t)-\frac{X(t)}{\tscale{t}}}
 +\frac{\abs{X(t)}}{\tscale{t}}
 \le C\tscale{t}^{2/21}\tlog{t}^{4/21}.
\]
Taking the supremum over the support proves
\eqref{eq:main-bound}.

Finally, the compact velocity support and the \(L^\infty\) bound for \(f\) give
\begin{equation}
 \norm{\rho(t)}_{L^\infty}\le CR_f(t)^3. \label{eq:rho-infty}
\end{equation}
Applying \cref{lem:coulomb-interpolation} with \(p=5/3\), and then using
\eqref{eq:rho53} and \eqref{eq:rho-infty},
\begin{align*}
 \norm{E(t)}_{L^\infty}
 &\le C\norm{\rho(t)}_{L^{5/3}}^{5/9}
 \norm{\rho(t)}_{L^\infty}^{4/9}\\
 &\le C\tscale{t}^{-1/3}R_f(t)^{4/3}.
\end{align*}
Substituting \eqref{eq:main-bound} and observing that
\[
 -\frac13+\frac43\frac{2}{21}=-\frac{13}{63},
 \qquad
 \frac43\frac{4}{21}=\frac{16}{63},
\]
we obtain \eqref{eq:field-decay}.
\end{proof}

\begin{remark}[Origin of the logarithmic exponent]\label{rem:log}
If one first bounds \(\abs{w}\) by \(CQ\) in the \(\cB_5\) estimate, the
remaining unweighted radial integral introduces an additional logarithm. In
\cref{lem:B5}, the factor \(\abs{w}\) is kept inside the integral and cancels
the factor \(\abs{w}^{-1}\) in the branch radius. The resulting radial
integral is of size \(Q\), which cancels the prefactor \(Q^{-1}\). This is why
the final logarithmic exponent is \(4/21\). The algebraic exponent \(2/21\)
is determined by the balance between the
\(\tscale{T}^{-5/7}\) near-region term and the \(\tscale{T}^{-1}\) crossing term.
\end{remark}

\begin{remark}[Logical dependencies]\label{rem:dependencies}
The proof uses three substantive estimates from the earlier theory: the
dispersive bounds \eqref{eq:w2}--\eqref{eq:rho53}, the velocity-support estimate \eqref{eq:chen-li-support}, and the adaptive-time
lower bound \eqref{eq:Delta-lower}. It also follows the
Chen--Li/Schaeffer characteristic-crossing architecture, including the
one-sided extension lemma \cite{ChenLi,SchaefferAsymptotic}. The ingredients
specific to the present argument
are the Coulomb-corrected labeled-energy identity
\eqref{eq:labeled-identity} and its weighted implementation: the factor
\(w\) is retained in the near-region integrations, the target is selected
through the running envelope, and the resulting estimate is closed by the
record-time iteration.
\end{remark}

\paragraph{Data/Code Availability.}
The manuscript has no associated data or code.

\paragraph{Conflict of interest.}
The author declares that he has no conflict of interest.


\begin{thebibliography}{99}

\bibitem{ChaeHa}
M.~S. Chae and S.-Y. Ha,
\newblock New Lyapunov functionals of the Vlasov--Poisson system,
\newblock \emph{SIAM J. Math. Anal.} \textbf{37} (2006), 1709--1731.

\bibitem{ChenLi}
Z.~Chen and X.~Li,
\newblock Asymptotic growth of support and uniform decay of moments for the
Vlasov--Poisson system,
\newblock \emph{SIAM J. Math. Anal.} \textbf{50} (2018), 4180--4202.

\bibitem{ChenZhang}
Z.~Chen and X.~Zhang,
\newblock Sub-linear estimate of large velocities in a collisionless plasma,
\newblock \emph{Commun. Math. Sci.} \textbf{12} (2014), 279--291.

\bibitem{ChenZhangMoments}
Z.~Chen and X.~Zhang,
\newblock Global existence to the Vlasov--Poisson system and propagation of
moments without assumption of finite kinetic energy,
\newblock \emph{Comm. Math. Phys.} \textbf{343} (2016), 851--879.

\bibitem{Horst}
E.~Horst,
\newblock On the asymptotic growth of the solutions of the Vlasov--Poisson
system,
\newblock \emph{Math. Methods Appl. Sci.} \textbf{16} (1993), 75--85.

\bibitem{IllnerRein}
R.~Illner and G.~Rein,
\newblock Time decay of the solutions of the Vlasov--Poisson system in the
plasma physical case,
\newblock \emph{Math. Methods Appl. Sci.} \textbf{19} (1996), 1409--1413.

\bibitem{LionsPerthame}
P.-L.~Lions and B.~Perthame,
\newblock Propagation of moments and regularity for the three-dimensional
Vlasov--Poisson system,
\newblock \emph{Invent. Math.} \textbf{105} (1991), 415--430.

\bibitem{PallardNote}
C.~Pallard,
\newblock A note on the growth of velocities in a collisionless plasma,
\newblock \emph{Math. Methods Appl. Sci.} \textbf{34} (2011), 803--806.

\bibitem{PallardKRM}
C.~Pallard,
\newblock Growth estimates and uniform decay for a collisionless plasma,
\newblock \emph{Kinet. Relat. Models} \textbf{4} (2011), 549--567.

\bibitem{PallardJDE}
C.~Pallard,
\newblock Large velocities in a collisionless plasma,
\newblock \emph{J. Differential Equations} \textbf{252} (2012), 2864--2876.

\bibitem{Perthame}
B.~Perthame,
\newblock Time decay, propagation of low moments and dispersive effects for
kinetic equations,
\newblock \emph{Comm. Partial Differential Equations} \textbf{21} (1996),
659--686.

\bibitem{Pfaffelmoser}
K.~Pfaffelmoser,
\newblock Global classical solutions of the Vlasov--Poisson system in three
dimensions for general initial data,
\newblock \emph{J. Differential Equations} \textbf{95} (1992), 281--303.

\bibitem{ReinGrowth}
G.~Rein,
\newblock Growth estimates for the solutions of the Vlasov--Poisson system
in the plasma physics case,
\newblock \emph{Math. Nachr.} \textbf{191} (1998), 269--278.

\bibitem{ReinSurvey}
G.~Rein,
\newblock Collisionless kinetic equations from astrophysics---the
Vlasov--Poisson system,
\newblock in: C.~M. Dafermos and E. Feireisl (eds.),
\emph{Handbook of Differential Equations: Evolutionary Equations}, Vol.~3,
Elsevier/North-Holland, Amsterdam, 2007, pp.~383--476.

\bibitem{SchaefferGlobal}
J.~Schaeffer,
\newblock Global existence of smooth solutions to the Vlasov--Poisson system
in three dimensions,
\newblock \emph{Comm. Partial Differential Equations} \textbf{16} (1991),
1313--1335.

\bibitem{SchaefferAsymptotic}
J.~Schaeffer,
\newblock Asymptotic growth bounds for the Vlasov--Poisson system,
\newblock \emph{Math. Methods Appl. Sci.} \textbf{34} (2011), 262--277.

\end{thebibliography}
\end{document}